\documentclass[11pt]{amsart}

\usepackage{amssymb,amsmath,latexsym, amscd}
\input cyracc.def

\newtheorem{theorem}{Theorem}[section]
\newtheorem{lemma}[theorem]{Lemma}
\newtheorem{proposition}[theorem]{Proposition}
\newtheorem{definition}[theorem]{Definition}

\def\s{\sigma}

\begin{document}

\title{Exotic left orderings of the free groups from the Dehornoy ordering}

\date{\today}

\author[Adam Clay]{Adam Clay}
\address{Department of Mathematics\\
University of British Columbia \\
Vancouver \\
BC Canada V6T 1Z2} \email{aclay@math.ubc.ca}
\urladdr{http://www.math.ubc.ca/~aclay/} \maketitle

\begin{abstract}
We show that the restriction of the Dehornoy ordering to an appropriate free subgroup of the three-strand braid group defines a left ordering of the free group on $k$ generators, $k>1$, that has no convex subgroups. 
\end{abstract}

A group $G$ is said to be left-orderable if there exists a strict total ordering of its elements such that $g<h$ implies $fg<fh$ for all $f, g, h$ in $G$.  To each left ordering $<$ of a group $G$, we can associate the set $P = \{ g \in G | g>1 \}$, which is called the positive cone associated to the left ordering $<$.  The positive cone $P$ satisfies $P \cdot P \subset P$, and $P \sqcup P^{-1} \sqcup \{1 \} = G$.  Conversely, any subset $P$ satisfying these two properties defines a strict total ordering of the elements of $G$, via $g<h$ if and only if $g^{-1}h \in P$.  Any ordering defined in this way is easily seen to be invariant under left multiplication.   

We may strengthen our conditions on a left ordering $<$ of $G$ by requiring that for all $g, h >1$ in $G$, there must exist a positive integer $n$ such that $g<hg^n$.  In this case, the ordering is called Conradian (after the work of Conrad in \cite{PC59}).  It has since been observed that, equivalently, we may ask that this condition hold for $n=2$ \cite{NF07}.

Finally, the strongest condition we may require of an ordering $<$ of $G$ is that the ordering be invariant under multiplication from both sides, that is, $g<h$ implies $fg<fh$ and $gf<hf$ for all $f, g, h$ in $G$.  Equivalently, we may require that the positive cone associated to the ordering $<$ of $G$ be preserved by conjugation.  If either of these equivalent conditions is satisfied by the ordering $<$ of $G$, then the ordering is said to be a bi-ordering.

An important structure associated to a given left ordering $<$ of $G$ is the set of convex subgroups of $G$.  A subgroup $H \subset G$ is said to be convex in $G$  (with respect to the ordering $<$) if whenever $f, h$ are in $H$ and $g$ is in $G$, the implication $f<g<h \Rightarrow g \in H$ holds.  
 
Owing to work of Conrad and H\"older, the convex subgroups of bi-orderings and Conradian orderings are very well understood \cite{PC59}.  This leaves us with understanding the set of convex subgroups for the case of left orderings that are neither bi-orderings, nor Conradian orderings.  This problem seems to be quite difficult, as constructing Conradian orderings and bi-orderings of a group $G$ is in general somewhat easier than constructing left orderings of a group that are not Conrad orderings.

The two primary methods for constructing non-Conradian orderings of a group $G$ are given by the following proposition and theorem.

\begin{proposition}
\label{prop:con}
Let $G$ be a group, $K$ a subgroup of $G$ left ordered by the ordering $\prec$, and $G/K$ the set of left cosets of $K$ in $G$.   Suppose that $G/K$ is ordered by the ordering $\prec'$, satisfying $gK \prec' hK$ implies $fgK \prec' fhK$ for all $f, g, h$ in $G$.  Then a left ordering $<$ can be defined on $G$, according to the rule:  For every $g$ in $G$, $1<g$ if $g \in K$ and $1 \prec g$, or if $g \notin K$ and $K \prec' gK$. 
\end{proposition}

\begin{theorem}
\label{th:con}
 (Conrad, \cite{PC59}) A group $G$ is left-orderable if and only if $G$ acts effectively by order preserving automorphisms on a linearly ordered set.
\end{theorem}

In both of these cases, at least some of the convex subgroups of the constructed ordering are obvious.  In Proposition \ref{prop:con}, the subgroup $K$ is a convex subgroup in the left ordering $<$ of $G$.  In Theorem \ref{th:con}, the stabilizers under the $G$-action of points in the given linearly ordered set correspond to convex subgroups (see \cite{PC59} or \cite{NF07} for details of the construction). In light of the fact that both of these known methods for producing left orderings of a group result in an ordering that (often) contain convex subgroups, it is quite surprising to find that some non-Conradian left orderings may contain no proper, nontrivial convex subgroups whatsoever.  In this paper, we will left order the free groups of finite rank in a way so that the free group contains no proper, nontrivial convex subgroups with respect to our constructed ordering.  The construction relies heavily on the Dehornoy ordering of the braid group $B_3$.

The existence and a construction of such orderings of the free groups seems to have appeared only in \cite{SM85}.  The construction there, unlike our present setting, deals with creating a very unusual effective action on the rationals.  Our present approach is in simpler algebraic terms.

It is also worth noting that admitting a Conradian or bi-ordering that has no proper, nontrivial convex subgroups is a very restrictive condition on the group $G$, as the following theorem shows.
\begin{theorem} \cite{PC59}
Suppose that $G$ admits a Conradian or bi-ordering which has no proper, nontrivial convex subgroups.  Then $G$ is a subgroup of $(\mathbb{R}, +)$. 
\end{theorem}

In the case that $G$ admits a non-Conradian left ordering having no proper, nontrivial convex subgroups, it is not likely that the structure of $G$ must be so restricted.  While we will see that free groups admit such left orderings, there are also non-free, non-Abelian groups that admit such left orderings as well (\cite{MR77} Example 7.2.3).  It has also recently been shown that the braid groups themselves admit many left-orderings with no convex subgroups, see \cite{NW09}.

\section{A left ordering of $F_2$ having no convex subgroups.}
As a warm up for the general case, which will be slightly more involved, we deal first with the free group on two generators.

We begin by defining the Dehornoy left ordering of the braid groups (also known as the `standard' ordering), whose positive cone we shall denote $P_D$ \cite{DDRW08}, \cite{PD94}. Recall that for each integer $n \ge 2$, the Artin braid group $B_n$ is the group generated by 
$\s_1 , \s_2 , \dots , \s_{n-1}$, subject to the relations 
$$\s_i\s_j = \s_j\s_i {\rm \: if \:} |i-j| >1,\quad \s_i\s_j\s_i=\s_j\s_i\s_j {\rm \: if \: } |i-j| =1.$$  

\begin{definition}Let $w$ be a word in the generators $\s_i, \cdots , \s_{n-1}$.  Then $w$ is said to be: $i$-positive if the generator $\s_i$ occurs in $W$ with only positive exponents, $i$-negative if $\s_i$ occurs with only negative exponents, and $i$-neutral if $\s_i$ does not occur in $w$.
\end{definition}

We then define the positive cone of the Dehornoy ordering as

\begin{definition} 
The positive cone $P_D \subset B_n$ of the Dehornoy ordering is the set
\[P_D = \{ \beta \in B_n : \mbox{ $\beta$ is $i$-positive for some $i \leq n-1$}\}.\]
\end{definition}

Let $\beta \in B_n$ be any braid. An extremely important property of this ordering is that the conjugate $\beta \s_k \beta^{-1}$ is always $i$-positive for some $i$, for every generator $\s_k$ in $B_n$.  This property is referred to as the subword property \cite{DDRW08}.

Recall that the commutator subgroup $[B_3, B_3]$ is isomorphic to the free group $F_2$ on two generators.  The commutator subgroup is generated by the braids  $\beta_1 = \s_2 \s_1^{-1}$ and $\beta_2=\s_1 \s_2 \s_1^{-2}$ \cite{MR06}.  Of course we can instead take as generators the $1$-positive braids $\beta_1^{-1} = \s_1 \s_2^{-1}$ and $\beta_2^{-1} \beta_1^{-1} = \s_1^2 \s_2^{-2}$.  

Define a positive cone $P \subset F_2$ by $P = [B_3, B_3] \cap P_D$, with associated ordering $<$ of $F_2$.  Thus, the ordering $<$ of $F_2$ is the restriction of the Dehornoy ordering $<_D$ of $B_3$ to the (free) commutator subgroup $[B_3, B_3]$. 

\begin{theorem}The ordering $<$ of $F_2$ has no proper, non-trivial convex subgroups.
\label{th:1}
\end{theorem}
\begin{proof}
Let $C \subset F_2 = [B_3, B_3]$ be a nontrivial convex subgroup.  Then we may choose $1 < \beta \in F_2$ that is $1$-positive (no nontrivial $1$-neutral braids lie in $[B_3, B_3]$, because they do not have zero total exponent).  

There are now two cases to consider.
  
\textbf{Case 1.}
 Suppose that $\beta$ commutes with $\s_2$.  Then $\beta = \Delta_3^{2p} \s_2^q$ for some integers $p, q$ (\cite{FRZ96}, here $\Delta = \s_1 \s_2 \s_1$).  Since $\beta \in [B_3, B_3]$, we know that $q = -6p$, since $\beta$ must have zero total exponent, and $p >0$ because we have chosen $\beta$ to be $1$-positive. Then we have that $\Delta_3^2 < \Delta_3^{4p} \s_2^{-12p} = \beta^2$, so that $\beta$ is cofinal in the Dehornoy ordering \cite{DDRW08}.  Therefore, there exist integers $k, l$ so that in $F_2$ we have
\[ 1< \s_1 \s_2^{-1} < \beta^k, \mbox{ and } 1 < \s_1^2 \s_2^{-2} < \beta^l,\]
and thus $\s_1 \s_2^{-1}, \s_1^2 \s_2^{-2} \in C$ by convexity.  Therefore we must have $C = F_2$, as $C$ contains both generators of $F_2$.

\textbf{Case 2.} Suppose that $\beta$ and $\s_2$ do not commute.  Let $k>0$, and observe that $\beta \s_2^k \beta^{-1}$ is a $1$-positive braid by the subword property, so that the commutator $\beta \s_2^k \beta^{-1} \s_2^{-k}$ is also $1$-positive.   Next, because $\beta$ is $1$-positive, the braid $\s_2^k \beta^{-1} \s_2^{-k}$ is $1$-negative, so that $\s_2^k \beta^{-1} \s_2^{-k} < 1$, and thus $ \beta \s_2^k \beta^{-1} \s_2^{-k} < \beta$.  Thus, we have shown that $1< \beta \s_2^k \beta^{-1} \s_2^{-k} < \beta$, so that $\beta \s_2^k \beta^{-1} \s_2^{-k}$ must lie in the subgroup $C$, by convexity.

Now both the braids $\beta$ and $\beta \s_2^k \beta^{-1} \s_2^{-k}$ lie in the convex subgroup $C$, so the braid $ \s_2^k \beta^{-1} \s_2^{-k}$ (and hence its inverse  $ \s_2^k \beta \s_2^{-k}$) must also lie in $C$, for any choice of positive integer $k$. 

We now refine our choice of braid $\beta \in C$.  Suppose that $\beta$ is represented by the $1$-positive braid word $\s_2^u \s_1 w$, where $u$ is any integer, and $w$ is a $1$-positive, $1$-neutral or empty word.  Choose $k>0$ so that $u' = k + u >0$, and set $\beta' = \s_2^k \beta \s_2^{-k}$, so that $\beta'$ is represented by the $1$-positive braid word $\s_2^{u'} \s_1 w \s_2^{-k}$.  Note that $\beta' \in C$, from our work above.

We will now show that $C$ must contain both generators of $F_2$. Observe that the braid represented by the word $\s_2 \s_1^{-1} \s_2^{u'} \s_1 w \s_2^{-k}$ is $1$-positive, as  $\s_2 (\s_1^{-1} \s_2^{u'} \s_1) w \s_2^{-k} = \s_2 (\s_2 \s_1^{u'} \s_2^{-1}) w \s_2^{-k}$, and $u'>0$.  Therefore we have
\[ 1< \s_2 \s_1^{-1} \s_2^{u'} \s_1 w \s_2^{-k} \Rightarrow \s_1 \s_2^{-1} <  \s_2^{u'} \s_1 w \s_2^{-k} = \beta' \in C, \]
and since $1< \s_1 \s_2^{-1}$, this implies that $\s_1 \s_2^{-1} \in C$ by convexity.

Considering the second generator $\s_1^2 \s_2^{-2}$, observe that the braid represented by the word $\s_2^2 \s_1^{-2} \s_2^{u'} \s_1 w \s_2^{-k}$ is $1$-positive, as we compute
\[\s_2^2 \s_1^{-1} (\s_1^{-1} \s_2^{u'} \s_1) w \s_2^{-k} = \s_2^2 \s_1^{-1} (\s_2 \s_1^{u'} \s_2^{-1}) w \s_2^{-k}, \]
and
\[\s_2^2 (\s_1^{-1} \s_2 \s_1) \s_1^{u'-1} \s_2^{-1} w \s_2^{-k} = \s_2^2 (\s_2 \s_1 \s_2^{-1}) \s_1^{u'-1} \s_2^{-1} w \s_2^{-k},
\]
where $u' >0$.
Therefore we have 
\[1< \s_2^2 \s_1^{-2} \s_2^{u'} \s_1 w \s_2^{-k} \Rightarrow \s_1^2 \s_2^{-2} < \s_2^{u'} \s_1 w \s_2^{-k} = \beta' \in C,\]
and since $1< \s_1^2 \s_2^{-2}$, we conclude from convexity of $C$ that $\s_1^2 \s_2^{-2} \in C$.

Thus, $C$ contains both generators of $F_2$, so that $C = F_2$.

\end{proof}

\section{Left ordering the free groups of rank greater than two}

We now extend our results to cover those free groups $F_k$ with $k>2$.   Let $x = \s_1 \s_2^{-1}$ and $y = \s_1^2 \s_2^{-2}$ denote the generators of $F_2$, and we let $K_{n}$ denote the kernel of the map $F_2 \rightarrow \mathbb{Z}_{n-1}$ defined by $y \mapsto 0$, $x \mapsto 1$.  Here $\mathbb{Z}_{n-1}$ is the cyclic group of order $n-1$.   We will employ a proof very similar to that of Theorem \ref{th:1}, by considering $K_n \subset F_2 = [B_3, B_3]$, and showing that the restriction of the Dehornoy ordering to $K_n$ has no convex subgroups.  First we need to find a generating set for $K_n$.

\begin{lemma}
\label{lem:gen}
The subgroup $K_{n}$ is free of rank $n$, with basis \[y, x^{n-1}, xyx^{n-2}, x^2yx^{n-3}, \cdots , x^{n-2}yx.\]
\end{lemma}
\begin{proof}From Lemma 7.56 of \cite{JR94}, we know that $K_{n}$ is finitely generated.  Moreover, we may compute a generating set of $K_n$ as follows:  Consider the generating set $g_1 = x, g_2 = x^{-1}, g_3 = y, g_4= y^{-1}$ of $F_2$, and let $1, x, x^2, \cdots ,x^{n-2}$ be representatives of the right cosets of $K_n \subset F_2$.  For all $i, j$, there exists $h_{ij}$ and some coset representative $x^{k(i,j)}$ such that we may write $x^i g_j = h_{ij} x^{k(i,j)}$.  The elements $h_{ij}$ form a generating set for $K_n$.

In our present setting, we find for $i < n-2$
\[ x^i g_1 = x^{i} \cdot x = 1 \cdot x^{i+1},
\]
so that $h(i, 1) = 1$, and for $i = n-2$ we get $h(i, 1) = x^{n-1}$.  Similarly, we compute for $i \geq 1$ that
\[ x^i g_2  = x^{i} \cdot x^{-1} = 1 \cdot x^{i-1},
\]
so that $h(i, 2) =1$, and for $i=0$ we compute $h(i, 2) = x^{-(n-1)}$.

Next, for all $i$ we compute
\[ x^i y^{\pm 1} = x^i y^{\pm 1} x^{-i} \cdot x^i,
\]
so that $h(i, 3) = h(i, 4)^{-1} = x^i y x^{-i}$.
Eliminating inverses from this generating set yields the set 
\[y, x^{n-1}, xyx^{-1}, x^2yx^{-2}, \cdots , x^{n-2}yx^{-(n-2)}.\]
From Proposition 3.9 of \cite{LS77} we deduce that $K_n$ is of rank $n$, and therefore the generating set above must provide a basis for $K_n$.  Right multiplying those generators of the form $x^i y x^{-i}$ by the generator $x^{n-1}$ yields the desired generating set.
\end{proof}

Also important in the proof of Theorem \ref{th:1} was the action of conjugation by $\s_2$.  In order to generalize our theorem, we must make the following analysis.

Let $F_2$ be the free group on two generators $x$ and $y$, and define an automorphism $\phi: F_2 \rightarrow F_2$ according to the forumlas $\phi(x) = xy^{-1}x$, and $\phi(y) = xy^{-1}x^{2}$.  Then the following holds.

\begin{lemma}
Consider $F_2$ as the commutator subgroup $[B_3, B_3]$ with generators  $x=\s_1 \s_2^{-1}$ and $y=\s_1^2 \s_2^{-2}$. Then the automorphism $\phi$ of $F_2$ corresponds to conjugation of $[B_3, B_3]$ by the generator $\s_2 \in B_3$, so that $\phi(g) = \s_2^{-1} g \s_2$ for all $g \in F_2$.
\label{lem:conform}
\end{lemma}
\begin{proof}
The proof is computational.  First conjugating the generator $x$, we compute
\begin{eqnarray*}
\phi(x) & = & xy^{-1}x \nonumber \\
 & = & \s_1 \s_2^{-1} \s_2^2 \s_1^{-2} \s_1 \s_2^{-1} \nonumber \\
 & = &  (\s_1 \s_2 \s_1^{-1}) \s_2^{-1} \nonumber \\
 & = &   (\s_2^{-1} \s_1 \s_2) \s_2^{-1} \nonumber \\
& = &   \s_2^{-1} \s_1 \s_2^{-1} \s_2  \nonumber \\
& = &  \s_2^{-1} x \s_2  \nonumber \\
\end{eqnarray*}
and
\begin{eqnarray*}
\phi(y) & = & xy^{-1}x^2 \nonumber \\
 & = &  \s_1 \s_2^{-1} \s_2^2 \s_1^{-2} \s_1 \s_2^{-1} \s_1 \s_2^{-1} \nonumber \\
 & = &  (\s_1 \s_2 \s_1^{-1}) \s_2^{-1} \s_1 \s_2^{-1} \nonumber \\
 & = &  (\s_2^{-1} \s_1 \s_2) \s_2^{-1} \s_1 \s_2^{-1}  \nonumber \\
& = &  \s_2^{-1} \s_1^2 \s_2^{-1}  \nonumber \\
& = &  \s_2^{-1} \s_1^2 \s_2^{-2} \s_2  \nonumber \\
& = &  \s_2^{-1} y \s_2.  \nonumber \\
\end{eqnarray*}
\end{proof}

This computation allows us to show that $K_n$ is fixed by the conjugation action of $\s_2^6$ or $\s_2^{-6}$ on the commutator subgroup $[B_3, B_3]$.

\begin{lemma}
\label{lem:fix}
Let $\phi : F_2 \rightarrow F_2$ be the map arising from conjugation of $[B_3, B_3]$ by $\s_2$, namely $\phi(x) = xy^{-1}x,$ and $\phi(y) = xy^{-1}x^2$.  Then for all $n$, $\phi^6(K_n) = K_n$.
\end{lemma}
\begin{proof}
Consider the abelianization $F_2 \stackrel{ab}{\rightarrow} \mathbb{Z} \oplus \mathbb{Z}$. We find that $\phi$ descends to a map $\phi_*: \mathbb{Z}\oplus \mathbb{Z} \rightarrow  \mathbb{Z}\oplus \mathbb{Z}$, and that relative to the basis $\{ ab(x), ab(y) \}$ the map $\phi_*$ is represented by the matrix
\[ \left( \begin{array}{cc}
 2 & 3 \\
 -1 & -1  \end{array} \right).\] 
The sixth power of this matrix is the identity.  It follows that for any normal subgroup $K$ such that $F_2/K$ is abelian, we have $\phi^6(K) = K$. 
\end{proof}

Lastly, note that any generator of $K_n$, when we substitute $x=\s_1 \s_2^{-1}$ and $y=\s_1^2 \s_2^{-2}$, yields a product of braid group generators of the form $\s_1^{l_1} \s_2^{k_1} \s_1^{l_2} \cdots \s_2^{k_{m-1}} \s_1^{l_m} \s_2^{k_m}$, where $k_i<0$ and $l_i >0$ for all $i$.  Therefore, we require the following lemma in order to compare the generators to different braids in $K_n$.

\begin{lemma}
\label{lem:pos}
Any braid represented by a word of the form 
\[ \s_2^{k_1} \s_1^{l_1} \cdots \s_2^{k_m} \s_1^{l_m} \s_2^n \s_1,  \]
where $k_i > 0 $, $l_i <0$ for all $i$, and $n > 1$, is $1$-positive. 
\end{lemma}
\begin{proof}We use induction on $m$, the length of the product.  For $m=0$, the claim is trivial.  
Assuming the claim holds for those products of length $m-1$, we use the identities $\s_1^k \s_2 \s_1 = \s_2 \s_1 \s_2^k$ and $\s_1^{-1} \s_2^k \s_1 = \s_2 \s_1^k \s_2^{-1}$, and compute that
\begin{eqnarray*}
 \s_2^{k_1} \s_1^{l_2} \cdots \s_2^{k_m} \s_1^{l_m} \s_2^n \s_1 & = & \s_2^{k_1} \s_1^{l_2} \cdots \s_2^{k_m} \s_1^{l_m+1}(\s_1^{-1} \s_2^n \s_1) \nonumber \\
 & = &    \s_2^{k_1} \s_1^{l_2} \cdots \s_2^{k_m} \s_1^{l_m+1}(\s_2 \s_1^n \s_2^{-1}) \nonumber \\
 & = &  \s_2^{k_1} \s_1^{l_2} \cdots \s_2^{k_m} (\s_1^{l_m+1}\s_2 \s_1)\s_1^{n-1} \s_2^{-1} \nonumber \\
 & = &  \s_2^{k_1} \s_1^{l_2} \cdots \s_2^{k_m} (\s_2 \s_1 \s_2^{l_m+1})\s_1^{n-1} \s_2^{-1}  \nonumber \\
& = &   \s_2^{k_1} \s_1^{l_2} \cdots \s_2^{k_m+1} \s_1 (\s_2^{l_m+1}\s_1^{n-1} \s_2^{-1}).\\
\end{eqnarray*}
The bracketed expression $\s_2^{l_m+1}\s_1^{n-1} \s_2^{-1}$ is $1$-positive as $n>1$, and the remaining terms in the product above are representative of a $1$-positive braid, by assumption.   By induction, the claim is proven.
\end{proof}

\begin{theorem}
Let $n>2$. Then the restriction of the Dehornoy ordering to the subgroup $K_n \subset F_2 = [B_3, B_3]$ has no proper, nontrivial convex subgroups.
\end{theorem}
\begin{proof}
We proceed similarly to Theorem \ref{th:1}.  Suppose that $C \subset K_n$ is a nontrivial, convex subgroup, and let $\beta \in C$ be a $1$-positive braid.  Denote the generators of $K_n$ by $g_1, g_2, \cdots , g_n$, from Lemma \ref{lem:gen} we know that $g_i >1$ for all $i$.    There are two cases to consider.

\textbf{Case 1.} The braid $\beta$ commutes with $\s_2$.  In this case, we proceed as in Case 1 of Theorem \ref{th:1}, to conclude that $\beta$ must be cofinal in the Dehornoy ordering.  Thus, we can find an integer $k$ so that $\beta^k > g_i > 1$ for every generator $g_i$ of $K_n$.  Then $g_i \in C$ for all $i$, and we conclude $C = K_n$.

\textbf{Case 2.} Suppose that $\beta$ and $\s_2$ do not commute, and we proceed as in Case 2 of Theorem \ref{th:1}.  Then, by the subword property of the Dehornoy ordering, we know that $\beta \s_2^k \beta^{-1} > 1$ for all $k>0$, and hence  $\beta \s_2^k \beta^{-1} \s_2^{-k} >1$ as well.  We deduce that $1< \beta \s_2^k \beta^{-1} \s_2^{-k} < \beta$ for all $k >0$ as before.    However, the braid $\beta \s_2^k \beta^{-1} \s_2^{-k}$ is not necessarily an element of $K_n$, but as conjugation by $\s_2^6$ preserves $K_n$ by Lemma \ref{lem:fix}, we have $ \beta \s_2^{6k} \beta^{-1} \s_2^{-6k} \in K_n$ for all $k >0$.  Hence, the inequality $1< \beta \s_2^k \beta^{-1} \s_2^{-k} < \beta$ yields $ \beta \s_2^{6k} \beta^{-1} \s_2^{-6k} \in C$ for all $k>0$.   We conclude that $ \s_2^{6k} \beta^{-1} \s_2^{-6k} \in C$ for all $k>0$.

 Proceeding as in the proof of Theorem \ref{th:1}, we may conjugate $\beta$ by an appropriate (sixth) power of $\s_2$ to conclude that the convex subgroup $C$ in $K_n$ contains a braid represented by a word of the form  $\s_2^u \s_1 w$, where $u>1$, and $w$ is a $1$-positive, $1$-neutral or empty word.  Then for each generator $g_i$ of $K_n$, consider the braid represented by the word $g_i^{-1} \s_2^u \s_1 w$.
As each $g_i$ contains only positive powers of the braids $x=\s_1 \s_2^{-1}$ and $y=\s_1^2 \s_2^{-2}$, we see that $g_i^{-1} \s_2^u \s_1$ represents a $1$-positive braid, by Lemma \ref{lem:pos}.  Therefore, the braid $g_i^{-1} \s_2^u \s_1 w$ is $1$-positive, and we conclude that $1 < g_i < \s_2^u \s_1 w \in C$, hence $g_i \in C$ for all $i$, and $C = K_n$.
\end{proof}

\bibliographystyle{plain}
\bibliography{candidacy}

\def\cprime{$'$}
\begin{thebibliography}{10}

\bibitem{MR77}
Roberta Botto~Mura and Akbar Rhemtulla.
\newblock {\em Orderable groups}.
\newblock Marcel Dekker Inc., New York, 1977.
\newblock Lecture Notes in Pure and Applied Mathematics, Vol. 27.

\bibitem{PC59}
Paul Conrad.
\newblock Right-ordered groups.
\newblock {\em Michigan Math. J.}, 6:267--275, 1959.

\bibitem{PD94}
Patrick Dehornoy.
\newblock Braid groups and left distributive operations.
\newblock {\em Trans. Amer. Math. Soc.}, 345(1):115--150, 1994.

\bibitem{DDRW08}
Patrick Dehornoy, Ivan Dynnikov, Dale Rolfsen, and Bert Wiest.
\newblock {\em Ordering Braids}, volume 148 of {\em Surveys and Monographs}.
\newblock American Mathematical Society, Providence, RI, 2008.

\bibitem{FRZ96}
Roger Fenn, Dale Rolfsen, and Jun Zhu.
\newblock Centralisers in the braid group and singular braid monoid.
\newblock {\em Enseign. Math. (2)}, 42(1-2):75--96, 1996.

\bibitem{LS77}
Roger~C. Lyndon and Paul~E. Schupp.
\newblock {\em Combinatorial group theory}.
\newblock Classics in Mathematics. Springer-Verlag, Berlin, 2001.
\newblock Reprint of the 1977 edition.

\bibitem{SM85}
Stephen~H. McCleary.
\newblock Free lattice-ordered groups represented as {$o$}-{$2$} transitive
  {$l$}-permutation groups.
\newblock {\em Trans. Amer. Math. Soc.}, 290(1):69--79, 1985.

\bibitem{MR06}
Jamie Mulholland and Dale Rolfsen.
\newblock Local indicability and commutator subgroups of artin groups.
\newblock Preprint., available via http://arxiv.org/abs/math/0606116.

\bibitem{NF07}
Andr\'{e}s Navas.
\newblock On the dynamics of (left) orderable groups.
\newblock Preprint, available via http://arXiv.org/pdf/gr-qc/040406.

\bibitem{NW09}
Andr\'{e}s Navas and Bert Wiest.
\newblock Nielsen-thurston orderings and the space of braid orderings.
\newblock Preprint., available via http://arxiv.org/pdf/0906.2605.

\bibitem{JR94}
Joseph~J. Rotman.
\newblock {\em An introduction to the theory of groups}, volume 148 of {\em
  Graduate Texts in Mathematics}.
\newblock Springer-Verlag, New York, fourth edition, 1995.

\end{thebibliography}

\end{document}